\documentclass[11pt]{amsart}
\usepackage{calc,amssymb,amsthm,amsmath}
\RequirePackage[dvipsnames,usenames]{xcolor}
\usepackage{hyperref}
\hypersetup{
bookmarks,
bookmarksdepth=3,
bookmarksopen,
bookmarksnumbered,
pdfstartview=FitH,
colorlinks,backref,hyperindex,
linkcolor=Sepia,
anchorcolor=BurntOrange,
citecolor=MidnightBlue,
citecolor=OliveGreen,
filecolor=BlueViolet,
menucolor=Yellow,
urlcolor=OliveGreen
}
\usepackage{alltt}
\usepackage{multicol}
\usepackage{etex}
\usepackage{xspace}
\newcommand{\ie}{{\itshape i.e.} }
\usepackage[active]{srcltx}
\usepackage{stmaryrd}
\usepackage{tikz}
\usepackage[mathscr]{euscript}
\usepackage{calligra,mathrsfs}

\usepackage[left=1.2in,top=1.0in,right=1.2in,bottom=1.0in]{geometry}

\usepackage{verbatim}
\usepackage{enumerate}

\usepackage{marginnote}

\newtheorem{theorem}{Theorem}[section]
\newtheorem{lemma}[theorem]{Lemma}
\newtheorem{proposition}[theorem]{Proposition}
\newtheorem{corollary}[theorem]{Corollary}

\theoremstyle{definition}

\theoremstyle{remark}
\newtheorem{remark}[theorem]{Remark}

\newtheorem{question}[theorem]{Question}
\numberwithin{equation}{subsection}

\numberwithin{equation}{theorem}

\DeclareMathOperator{\en}{end}

\DeclareMathOperator{\fm}{\mathfrak{m}}

\DeclareMathOperator{\fp}{\mathfrak{p}}
\DeclareMathOperator{\fa}{\mathfrak{a}}
\usepackage{graphicx}

\usepackage[all,cmtip]{xy}

\DeclareMathOperator{\Hom}{Hom} 
\DeclareMathOperator{\Ext}{Ext}
\DeclareMathOperator{\Tor}{Tor} 

\DeclareMathOperator{\Proj}{Proj}
\DeclareMathOperator{\Sym}{Sym}

\DeclareMathOperator{\cC}{\mathcal{C}}
\DeclareMathOperator{\cF}{\mathcal{F}}

\DeclareMathOperator{\sO}{\mathscr{O}}
\DeclareMathOperator{\sI}{\mathscr{I}}
\DeclareMathOperator{\sN}{\mathscr{N}}

\DeclareMathOperator{\reg}{reg}

\DeclareMathOperator{\PP}{\mathbb{P}}

\DeclareMathOperator{\E}{E}
\DeclareMathOperator{\beg}{beg}

\DeclareMathOperator{\Soc}{{Soc}}
\DeclareMathOperator{\sHom}{\mathscr{H}\text{\kern -3pt {\calligra\large om}}\,}

\renewcommand{\phi}{\varphi}

\renewcommand{\theta}{\vartheta}

\renewcommand{\epsilon}{\varepsilon}

\renewcommand{\to}[1][]{\xrightarrow{\ #1\ }}

\makeatletter
\@namedef{subjclassname@2020}{%
  \textup{2020} Mathematics Subject Classification}
\makeatother

\begin{document}

\title{On asymptotic socle degrees of local cohomology modules}

\author{Wenliang Zhang}
\address{Department of Mathematics, Statistics, and Computer Science, University of Illinois, Chicago, IL 60607, USA}
\email{wlzhang@uic.edu}
\begin{abstract}
Let $R$ be a standard graded algebra over a field $k$ and $I$ be a homogeneous ideal of $R$. We study the question whether there is a constant $c$ such that $\Soc(H^{j}_{\fm}(R/I^t))_{<-ct}=0$ for all $t\geq 1$ and a variation of this question. We also draw a connection between this question and the notion of gauge-boundedness in prime characteristic.
\end{abstract}

\thanks{The author is partially supported by the NSF through DMS-1752081.}
\subjclass[2020]{13D45, 14B15}
\maketitle
\section{Introduction}

Let $k$ be a field and let $R$ be a standard graded $k$-algebra, {\ie} $R=k[x_1,\dots,x_n]/\fa$ where $\fa$ is a homogeneous ideal in $k[x_1,\dots,x_n]$ that is standard graded. Let $\fm=(x_1,\dots,x_n)$ denote the homogeneous maximal ideal of $R$. Following the notation in \cite[\S13.1]{BrodmannSharpBook}, for each graded $R$-module $M$, we set \
\[\beg(M):=\inf\{i\mid M_i\neq 0\}\qquad  \qquad \en(M):=\sup\{i\mid M_i\neq 0\}.\]
We will denote the socle of $M$ by $\Soc(M)$, {\it i.e.} 
\[\Soc(M):=\{z\in M\mid \fm z=0\}=\Hom_R(R/\fm,M).\] 
When $M$ is artinian, $\Soc(M)$ is a finite-dimensional $k$-vector space and hence both $\beg(\Soc(M))$ and $\en(\Soc(M))$ are finite. When $M$ is a finitely generated graded $R$-module, the local cohomology modules $H^j_{\fm}(M)$ are artinian; consequently both $\beg(\Soc(H^j_{\fm}(M)))$ and $\en(\Soc(H^j_{\fm}(M)))$ are finite. It is clear that
\[\en(\Soc(H^j_{\fm}(M)))=\en(H^j_{\fm}(M))\]
and hence the largest socle degrees of $H^j_{\fm}(M)$ calculate the Castelnuovo-Mumford regularity of $M$:
\[\reg(M)=\max\{\en(\Soc(H^j_{\fm}(M)))+j\}.\]
Let $I$ be a homogeneous ideal of $R$. Then it is known that there exists a constant $c$ such that $\en(H^j_{\fm}(R/I^t))\leq ct$ ({\it cf.} \cite{CutkoskyHerzogTrung99, Kodiyalam2000, TrungWangLinearRegularity}). When one changes $\{I^t\}$ to other descending chains of ideals $\{I_t\}$ that are cofinal with $\{I^t\}$ ({\it e.g.} symbolic powers or Frobenius powers), then it remains an open problem whether $\en(H^j_{\fm}(R/I_t))$ can be bounded from above linearly with respect to $t$.

In this paper, we are mainly concerned with the lowest socle degree, especially the following question.
\begin{question}
\label{main question}
Let $I$ be a homogeneous ideal of $R$. Does there exist a constant $c$ such that 
\[\beg(\Soc(H^j_{\fm}(R/I^t))) \geq -ct\]
for all $j$ and all $t$?

Or, more generally, let $\{I_t\}$ be a descending chain of ideals that is cofinal with $\{I^t\}_{\geq 1}$. Is there a linear lower bound on $\beg(\Soc(H^j_{\fm}(R/I_t)))$?
\end{question}   

One ought to remark that in the second part of Question \ref{main question} `linear bound' needs to be interpreted appropriately according to $\{I_t\}_{\geq 1}$. For instance, if $I_t=I^{[p^t]}$ (the Frobenius powers in characteristic $p$), then `linear lower bound' means linear with respect to $p^t$. 

There are at least two sources of motivation behind Question \ref{main question}. One source stems from the recent results in \cite[Theorem 1.2]{BBLSZ3} and \cite[Corollary 2.6]{DaoMontano} which says that under appropriate hypothesis on $R$ and $I$ there is a constant $c$ such that $H^j_{\fm}(R/I^t)_{\leq -ct}=0$ for $j<\dim(R/I)$. This result implies that $\beg(\Soc(H^j_{\fm}(R/I_t)))\geq -ct$ for $j<\dim(R/I)$ (under the same hypothesis on $R$ and $I$). Note that, a priori, \cite[Theorem 1.2]{BBLSZ3} and \cite[Corollary 2.6]{DaoMontano} have no bearing on the top local cohomology $H^{\dim(R/I)}_{\fm}(R/I^t)$ since it is non-zero in all sufficiently negative degrees. However, Question \ref{main question} is still valid for $\Soc(H^{\dim(R/I)}_{\fm}(R/I_t))$. Similarly, without particular hypotheses on $R$ and $I$, local cohomology modules $H^j_{\fm}(R/I^t)$ may not be finitely generated and hence it may be the case that $H^j_{\fm}(R/I^t)_{\ell}\neq 0$ for all $\ell\ll 0$. Question \ref{main question} remains valid without any further hypotheses on $R$ and $I$, and can certainly be viewed as a natural extension of the main results in \cite{DaoMontano} and \cite{BBLSZ3}. The other source comes from the connection between the notion of gauge-boundedness and a linear lower bound of $\Soc(H^d_{\fm}(\omega^{[p^e]}))$ ({\it cf.} Theorem \ref{gauge same as least socle degree}). Therefore a positive answer to Question \ref{main question} will imply the gauge-boundedness (and hence the discreteness of $F$-jumping numbers for test ideals) for a large class of graded rings ({\it cf.} \S4 for details). Note that the gauge-boundedness of the full Cartier algebra remains an open question for non-Gorenstein rings ({\it cf.} \cite[Question 5.4(a)]{BlickleTestIdealspInverseMaps}).

Our first main result connects the vanishing $H^j_{\fm}(R/I^t)_{\leq -ct}=0$ for all lower local cohomology and a linear lower bound on the socle degree of the top local cohomology.

\begin{theorem}[Theorem \ref{main criterion}]
\label{main criterion in intro}
Let $R$ be a standard graded $k$-algebra over a field $k$ and let $I$ be a homogeneous ideal. Set $d=\dim(R/I)$. Assume there is a constant $c'$ such that 
\begin{equation}
\label{linear assumption on lower lc}
\beg(\Ext^i_R(k, H^j_{\fm}(R/I^t)))\geq -c't
\end{equation} 
for $j<d$, $i<d+2$ and for all $t\geq 1$. Then there exists a constant $c$ such that
\[\beg(\Soc(H^d_{\fm}(R/I^t)))\geq -ct,\] 
for all $t\geq 1$.
\end{theorem}  

In order to apply this theorem in \S\ref{gauge-bounded}, we also prove the following asymptotic vanishing theorem.

\begin{theorem}[Theorem \ref{Cartier divisor case}]
\label{Cartier divisor case in intro}
Let $R$ be a standard graded $k$-algebra where $k$ is a field of arbitrary characteristic and $I$ be a homogeneous ideal of height one in $R$ such that $R/I$ is equi-dimensional. Assume that $R_{\fp}$ is Cohen-Macaulay and that $I_{\fp}$ is principally generated by a non-zero-divisor for each non-maximal prime ideal $\fp$ in $R$. Then there exists a constant $c$ such that
\[\beg(H^j_{\fm}(R/I^t))\geq -ct\]
for $j<\dim(R/I)$ and all $t\geq 1$. 
\end{theorem}

As a consequence of Theorem \ref{main criterion in intro}, we prove 
\begin{theorem}[Corollary \ref{corollary on socle of top local cohomology}]
Assume $R$ and $I$ satisfy the hypothesis in any one of the following three results
\begin{enumerate}
\item Theorem \ref{Cartier divisor case in intro}, or 
\item \cite[Theorem 1.2]{BBLSZ3}, or 
\item \cite[Corollary 2.6]{DaoMontano}.
\end{enumerate} 
Then Question \ref{main question} has a positive answer for $H^{\dim(R/I)}_{\fm}(R/I^t)$, {\it i.e.} there is a constant $c$ such that
\[\beg(\Soc(H^{\dim(R/I)}_{\fm}(R/I^t)))\geq -ct\]
for all $t\geq 1$.
\end{theorem}

Combining these two theorems together with our Theorem \ref{gauge same as least socle degree}, we have the following consequence on gauge-boundedness for standard graded rings with isolated non-Gorenstein points, which partially answers \cite[Question 5.4(a)]{BlickleTestIdealspInverseMaps}. 

\begin{theorem}[Theorem \ref{isolated Gorenstein gauge-bounded}]
\label{isolated Gorenstein in intro}
Let $R$ be a standard graded integral domain over an $F$-finite field $k$ of characteristic $p$. Assume that $R$ satisfies $(S_2)$ condition and that $R_{\fp}$ is Gorenstein for each non-maximal prime ideal $\fp$. Then the full Cartier algebra $\cC^R$ is gauge bounded.
\end{theorem}

The paper is organized as follows. In \S\ref{canonical ideal and matlis dual}, we collect some materials on graded canonical ideals and (graded) Matlis dual which will be needed in \S\ref{gauge-bounded}. In \S\ref{main technical results}, we prove some technical results including both Theorems \ref{main criterion in intro} and \ref{Cartier divisor case in intro}. In the last section \S\ref{gauge-bounded}, we draw connections between our Question \ref{main question} and the notion of gauge-boundedness and prove Theorem \ref{isolated Gorenstein in intro}.


\section{Some generalities on canonical ideals and Matlis dual}
\label{canonical ideal and matlis dual}

In this section we collect some basic facts on graded canonical ideals and Matlis dual. These facts might already be known to experts; we opt to include them here for lack of proper references.

Let $R$ be a $d$-dimensional standard graded $k$-algebra where $k$ is a field. We can write $R:=S/\fa$ where $S=k[x_1,\dots,x_n]$ be a polynomial ring with the standard grading and $I\fa$ is a homogeneous ideal in $S$. Set $\fm:=(x_1,\dots,x_n)$. Since $R$ is a quotient of a polynomial ring, it admits a (graded) canonical module $\Omega_R$ which is isomorphic to $\Ext^{n-d}_S(R,S(-n))$ (in the category of graded modules). We will denote the graded Matlis dual by $-()^{\vee}$, {\it i.e.} for each graded $R$-module $M$
\[(M^{\vee})_\ell=\Hom_k(M_{-\ell},k).\]
Part of the graded local duality says that, for each finitely generated graded $R$-module $M$, there is a degree-preserving isomorphism $H^d_{\fm}(M)\cong \Hom_R(M,\Omega_R)^{\vee}$. 

\begin{proposition}
\label{graded canonical ideal}
Let $R$ be a $d$-dimensional standard graded $k$-algebra where $k$ is a field. Assume that $R$ is an integral domain and satisfies Serre's ($S_2$) condition. Then 
\begin{enumerate}
\item there is a homogeneous ideal $\omega$ of $R$ such that a degree-shift $\omega(\beg(\omega)+a(R))$ is isomorphic to the graded canonical module $\Omega$, where $a(R)=\en(H^{d}_{\fm}(R))$ (called the $a$-invariant of $R$); and
\item the natural map $R\to \Hom_R(\omega, \omega)$ is a degree-preserving isomorphism.
\end{enumerate}
\end{proposition}
\begin{proof}
For part (a), the proof of the fact $\omega(a)$ is isomorphic to $\Omega$ (as graded modules) for some integer $a$ follows the same proof as the one of \cite[Proposition 2.4]{MaSufficientConditionFPurity}. We claim that $a=\beg(\omega)+a(R)$ and we reason as follows. By the graded local duality, $\omega(a)^\vee\cong H^d_{\fm}(R)$. Thus, $a(R)=\en(H^{d}_{\fm}(R))=-\beg(\omega(a))=-(\beg(\omega)-a)=-\beg(\omega)+a$. This implies that $a=\beg(\omega)+a(R)$. (Note that the ($S_2$) assumption is not used in part (a).)

For part (b), we use our ($S_2$) assumption and the proof is the same as the one of \cite[Proposition 4.4]{AoyamaBasicResultsCanonicalModules} (or \cite[Remrk 2.2(c)]{HochsterHunekeConnectedness}).
\end{proof}

The following observation on Matlis dual will be useful for us in \S\ref{gauge-bounded}.
\begin{proposition}
\label{Matlis dual gives essential extension}
Let $(A,\fm)$ be a noetherian complete local ring and let $M$ be a finitely generated $A$-module. Let $\E$ denote the injective hull of $A/\fm$. Assume that the surjection $\cF_0=A^{\oplus \mu}\twoheadrightarrow M$ induces an isomorphism $\cF_0/\fm F_0\xrightarrow{\sim}M/\fm M$. Then $\Hom_A(M,\E)\to \Hom_A(\cF_0, \E)$ (the Matlis dual of $\cF_0\twoheadrightarrow M$) is an essential extension.

Likewise, let $A$ be a standard graded ring over a field $k$ and $M$ be a finitely generated graded $A$-module. Assume that the degree-preserving surjection $\oplus_{i=1}^{\mu} A(\alpha_i)\twoheadrightarrow M$ induces an isomorphism $\oplus_{i=1}^{\mu} (A/\fm)(\alpha_i)\xrightarrow{\sim}M/\fm M$. Then the graded Matlis dual of $\oplus_{i=1}^{\mu} A(\alpha_i)\twoheadrightarrow M$ is an essential extension.
\end{proposition}
\begin{proof}
Since the proofs in the local case and in the graded case are identical, we include the proof in the local case only. 

Assume otherwise, $\Hom_A(\cF_0,\E)$ would admit a non-zero submodule $L$ such that $L\cap\Hom_A(M,\E)=0$, {\it i.e.} $L\oplus\Hom_A(M,\E)$ is naturally a submodule of $\Hom_A(\cF_0,\E)$. By Matlis Duality, we then would have
\[\cF_0\cong \Hom_A(\Hom_A(\cF_0,\E),\E)\twoheadrightarrow \Hom_A(L\oplus\Hom_A(M,\E),\E)\cong \Hom_A(L,\E)\oplus M\]
Since $\Hom_A(L,\E)\neq 0$, this would imply that the rank of $\cF_0$ is strictly greater than the minimal number of generators of $M$, a contradiction.
\end{proof}

The following corollary is immediate.
\begin{corollary}
\label{minimal injective resolution by Matlis dual}

Let $(A,\fm)$ be a noetherian complete local ring and let $M$ be a finitely generated $A$-module. Let $\E$ denote the injective hull of $A/\fm$. Let $\cF_{\bullet}\to M\to 0$ be the minimal resolution of $M$. Then 
\[
0\to \Hom_A(M,\E)\to \Hom_A(\cF_{\bullet}, \E)
\]
is the minimal injective resolution of $\Hom_A(M,\E)$. 

Likewise, if $A$ is a standard graded ring over a field $k$ and $M$ is a finitely generated graded $A$-module and $\cF_{\bullet}\to M\to 0$ be the graded minimal resolution of $M$. Then the graded Matlis dual of $\cF_{\bullet}\to M\to 0$ is the graded minimal injective resolution of the graded Matlis dual of $M$.
\end{corollary}

\section{Main technical results}
\label{main technical results}

Throughout this section, $R$ denotes a standard graded $k$-algebra over a field $k$ and $\cF_{\bullet}\to k\to 0$ denotes the minimal graded free resolution of $k$ where $\cF_i=\oplus_{\ell}R(-\alpha_{i\ell})$ with $\alpha_{i\ell} \geq 0$. We set $\alpha_i:=\max\{\alpha_{i\ell}\}$. For each finitely generated graded $R$-module $M$, we will denote the reflexive dual $\Hom_R(M,R)$\footnote{Since $M$ is finitely generated, $\Hom_R(M,R)$ coincides with the graded $\sideset{^*}{_R}{\Hom}(M,R)$.} by ${M^*}$. Analogously, for each coherent sheaf $\mathscr{F}$ on $X=\Proj(R)$, we set ${\mathscr{F}^*}:=\sHom(\mathscr{F},\sO_X)$.

We begin with the following observation which follows immediately from the definition of $\beg(-)$.
\begin{lemma}
\label{beg under subquotient}
Let $0\to L\to M\to N\to 0$ be a short exact sequence of graded $R$-modules where all morphisms are degree-preserving. Then
\[\beg(M)=\min\{\beg(L),\beg(N)\}.\]

In particular, if $M'$ is a subquotient of $M$, then $\beg(M')\geq \beg(M)$.
\end{lemma}

Our next lemma provides a lower bound of the lowest degree of $\Ext^i_R(k, M)$ in terms of $\beg(M)$ and the graded free resolution of $k$.
\begin{lemma}
\label{beg for Ext with k}
Let $M$ be a finitely generated graded $R$-module. Then
\[\beg(\Ext^i_R(k, M))\geq \beg(M)-\alpha_i.\] 
\end{lemma}
\begin{proof}
$\Ext^i_R(k, M)$ is the $i$-th homology of the complex $\Hom_R(\cF_{\bullet}, M)$ and hence is a subquotient of $\Hom_R(\cF_i,M)\cong \oplus_{\ell}M(\alpha_{i\ell})$. It is clear that $\beg(\oplus_{\ell}M(\alpha_{i\ell}))=\beg(M)-\alpha_i$, and hence Lemma \ref{beg under subquotient} implies $\beg(\Ext^i_R(k, M))\geq \beg(M)-\alpha_i$.
\end{proof}

One of our main results in this section is the following criterion.

\begin{theorem}
\label{main criterion}
Let $R$ be a standard graded $k$-algebra over a field $k$ and let $I$ be a homogeneous ideal. Set $d=\dim(R/I)$. Assume there is a constant $c'$ such that 
\begin{equation}
\label{linear assumption on lower lc}
\beg(\Ext^i_R(k, H^j_{\fm}(R/I^t)))\geq -c't
\end{equation} 
for $j<d$, $i<d+2$ and for all $t\geq 1$. Then there exists a constant $c$ such that
\[\beg(\Soc(H^d_{\fm}(R/I^t)))\geq -ct,\] 
for all $t\geq 1$.
\end{theorem} 
\begin{proof}
Since the $\fm$-torsion functor $\Gamma_{\fm}$ sends injectives to injectives and $\Hom_R(R/\fm, \Gamma_{\fm}(-))=\Hom_R(R/\fm,-)$, we have the following spectral sequence:
\[E^{i,j}_2:=\Ext^i_R(R/\fm, H^j_{\fm}(-))\Rightarrow \Ext^{i+j}_R(R/\fm, -).\]
Since $\beg(R/I^t)=0$ for all $t\geq 1$, by Lemma \ref{beg for Ext with k} we have
\[\beg(\Ext^d_R(k, R/I^t))\geq -\alpha_d.\]
Recall that $\alpha_d$ denotes the largest degree-shift in the $d$-th graded free module in the minimal graded free resolution of $k$. The convergence of this spectral sequence means that $\Ext^d_R(k, R/I^t)$ admits a filtration $\cdots\subseteq E_{i+1}\subseteq E_i\subseteq \cdots\subseteq \Ext^d_R(k, R/I^t)$ such that $E^{0,d}_{\infty}=E_0/E_1$. Since $E^{0,d}_{\infty}$ is a subqotient of $\Ext^d_R(k, R/I^t)$, it follows from Lemma \ref{beg for Ext with k} that
\[\beg(E^{0,d}_{\infty})\geq \beg(\Ext^d_R(k, R/I^t))\geq -\alpha_d.\]
Since $H^{\geq d+1}_{\fm}(R/I^t)=0$, we have $E^{0,d}_{d+2}=E^{0,d}_{\infty}$. For each $i\leq d+1$, the incoming differential to $E^{0,d}_{i}$ is always 0; the outgoing differential is $\delta^{0,d}_{i}:E^{0,d}_{i}\to E^{i, d-i+1}_{i}$. Therefore, $E^{0,d}_{i+1}=\ker(\delta^{0,d}_{i})\subseteq E^{0,d}_{i}$ and hence $E^{0,d}_{i}/E^{0,d}_{i+1}$ is a submodule of $E^{i, d-i+1}_{i}$.

Since $E^{i, d-i+1}_{2}=\Ext^i_R(k, H^{d-i+1}_{\fm}(R/I^t))$ and the module $E^{i, d-i+1}_{i}$ is a subquotient of $E^{i, d-i+1}_{i-1}$ for each $i$, combining our assumption (\ref{linear assumption on lower lc}) and Lemma \ref{beg under subquotient}, we have
\[\beg(E^{i, d-i+1}_{i})\geq -c't\]
for $i\leq d+1$ and all $t\geq 1$. Since $E^{0,d}_{i}/E^{0,d}_{i+1}$ is a submodule of $E^{i, d-i+1}_{i}$, Lemma \ref{beg under subquotient} implies that
\[\beg(E^{0,d}_{i}/E^{0,d}_{i+1})\geq -c't\]
for all $i\leq d+1$ and for all $t\geq 1$. Since $\beg(E^{0,d}_{d+2})=\beg(E^{0,d}_{\infty})\geq -\alpha_d$, Lemma \ref{beg under subquotient} implies that $\beg(E^{0,d}_{d+1})\geq \min\{-c't, -\alpha_d\}$. Now a reverse induction on $i$ shows that 
\[\beg(\Soc(H^d_{\fm}(R/I^t)))=\beg(E^{0,d}_2)\geq \min\{-c't, -\alpha_d\}\]
which completes the proof.
\end{proof}

\begin{remark}
One may adjust the proof of Theorem \ref{main criterion} to derive a sufficient condition on a linear lower bound of $\beg(\Ext^s_R(k, H^d_{\fm}(R/I^t)))$ for an integer $s>0$ (one would need a different range on the homological index $i$ in the hypothesis on linear lower bounds on $\beg(\Ext^i_R(k, H^j_{\fm}(R/I^t)))$). Since our main application in \S\ref{gauge-bounded} only requires a linear lower bound on the socle degrees of $H^d_{\fm}(R/I^t)$, we will leave the formulation of such a sufficient condition for $\beg(\Ext^s_R(k, H^d_{\fm}(R/I^t)))$ ($s>0$) to another project or another person.
\end{remark}

The following asymptotic vanishing result will be crucial for our application in the next section.
\begin{theorem}
\label{Cartier divisor case}
Let $R$ is a standard graded $k$-algebra where $k$ is a field of arbitrary characteristic and $I$ be a homogeneous ideal of height one in $R$ such that $R/I$ is equi-dimensional. Assume that $R_{\fp}$ is Cohen-Macaulay and that $I_{\fp}$ is principally generated by a non-zero-divisor for each non-maximal prime ideal $\fp$ in $R$. Then there exists a constant $c$ such that
\[\beg(H^j_{\fm}(R/I^t))\geq -ct\]
for $j<\dim(R/I)$ and all $t\geq 1$. 
\end{theorem}
\begin{proof}
Write $R=S/J$ where $S=k[x_1,\dots,x_n]$ and set $X:=\Proj(R/I)\subset \Proj(R)\subseteq \PP^{n-1}_k$. Set $X_t=\Proj(R/I^t)$, $\sI=\tilde{I}$, and $d=\dim(X)=\dim(R/I)-1$.

The local cohomology and sheaf cohomology are linked by an exact sequence 
\[0\to H^0_{\fm}(R/I^t)\to R/I^t\to \oplus_{\ell}H^0(X_t,\sO_{X_t}(\ell))\to H^1_{\fm}(R/I^t)\to 0\] 
and isomorphisms 
\[H^j_{\fm}(R/I^t)\cong \oplus_{\ell}H^{j-1}(X_t,\sO_{X_t}(\ell))\quad (j\geq 2),\]
where all maps involved are degree-preserving. Therefore, proving our theorem amounts to proving that there exists a constant $c$ such that 
\[H^{j-1}(X_t,\sO_{X_t}(\ell))=0\]
for all $\ell\leq -ct$, all $j\leq d$ and all $t\geq 1$.

The assumptions on $R$ and $I$ imply that $R/I$ is generalized Cohen-Macaulay. Therefore there exists a constant $c_0\geq 0$ such that 
\[\beg(H^j_{\fm}(R/I))\geq -c_0,\]
for $j\leq d$. And hence $H^{j-1}(X,\sO_{X}(\ell))=0$ for $j\leq d$ and for $\ell<-c_0$. Since $R/I$ is generalized Cohen-Macaulay, $X$ is a Cohen-Macaulay projective scheme. Let $\omega_X$ denote the canonical sheaf on $X$. 

Our assumption on $I$ implies that $\sI^t/\sI^{t+1}=\Sym^t(\sI/\sI^2)$. Set $\sN:={(\sI/\sI^2)^*}$ which is the normal bundle on $X$ (relative to $\Proj(R)$), then $\sI^t/\sI^{t+1}=\Sym^t({\sN^*})$. Combining the long exact sequence of sheaf cohomology associated with the short exact sequence 
\[0\to \Sym^t({\sN^*})\to \sO_{X_{t+1}}\to \sO_{X_t}\to 0\]
and an induction on $t$, we conclude that to prove our theorem it suffices to show there exists an integer $c_1$ such that 
\begin{equation}
\label{vanishing of dual normal bundle}
H^{j-1}(X,\Sym^t({\sN^*})(\ell))=0
\end{equation}
for $j\leq d$ and for $\ell<-c_1t$. By Serre duality (which holds since $X$ is Cohen-Macaulay and equi-dimensional), (\ref{vanishing of dual normal bundle}) is equivalent to  
\begin{equation}
\label{vanishing after Serre dual}
H^{d-j+1}(X, \Sym^t({\sN^*})^{*}(-\ell)\otimes \omega_X)=0
\end{equation}
for $j\leq d$ and for $\ell<-c_1t$ (\ie $-\ell>c_1t$).

By our assumption on $I$, the normal bundle $\sN$ is a line bundle. It is well-known that the symmetric powers agree with divided powers for line bundles ({\it e.g.} \cite[Lemma 2.17]{SatoTakagi}); hence 
\[\Sym^t({\sN^*})^{*}=\Sym^t(\sN).\]
Therefore, (\ref{vanishing after Serre dual}) is equivalent to 
\[H^{d-j+1}(X, \Sym^t(\sN)(-\ell)\otimes \omega_X)=0\]
for $j\leq d$ and for $\ell<-c_1t$ (\ie $-\ell>c_1t$). This amounts to finding a linear upper bound on the maximal non-vanishing degree which follows immediately from \cite[Theorem 2.1]{DaoMontano}. This completes the proof of our theorem.
\end{proof}

As an application of our criterion Theorem \ref{main criterion}, we have the following corollary.
\begin{corollary}
\label{corollary on socle of top local cohomology}
Assume $R$ and $I$ satisfy the hypothesis in one of the following three results
\begin{enumerate}
\item Theorem \ref{Cartier divisor case}, or 
\item \cite[Theorem 1.2]{BBLSZ3}, or 
\item \cite[Corollary 2.6]{DaoMontano}.
\end{enumerate} 
Then Question \ref{main question} has a positive answer for $H^{\dim(R/I)}_{\fm}(R/I^t)$, {\it i.e.} there is a constant $c$ such that
\[\beg(\Soc(H^{\dim(R/I)}_{\fm}(R/I^t)))\geq -ct\]
for all $t\geq 1$.
\end{corollary}
\begin{proof}
Under any of the hypothesis, there exists a constant $c'$ such that
\[\beg(H^j_{\fm}(R/I^t))\geq -c't\]
for all $t\geq 1$ and $j<\dim(R/I)$. Consequently,
\[\beg(\Soc(H^j_{\fm}(R/I^t)))\geq \beg(H^j_{\fm}(R/I^t))\geq -c't,\]
for all $t\geq 1$ and $j<\dim(R/I)$. Combining this with Lemma \ref{beg for Ext with k}, we see that the hypothesis in Theorem \ref{main criterion} are satisfied. This completes the proof.
\end{proof}

We end this section with the following lemma which is needed in the next section.
\begin{lemma}
\label{linking ideal and quotient}
Let $R$ be a $d$-dimensional standard graded algebra over a field $k$. Assume that $H^{d-1}_{\fm}(R)$ is finitely generated. Then for each homogeneous ideal $I$ the following statements are equivalent:
\begin{enumerate}
\item there is a constant $c$ such that $\beg\Soc(H^{d-1}_{\fm}(R/I^t))\geq -ct$ for all $t\geq 1$, 
\item there is a constant $c'$ such that $\beg\Soc(H^{d}_{\fm}(I^t))\geq -c't$ for all $t\geq 1$.
\end{enumerate}
\end{lemma}
\begin{proof}
Consider the exact sequence
\[\cdots \to H^{d-1}_{\fm}(R)\xrightarrow{f} H^{d-1}_{\fm}(R/I^t) \xrightarrow{g} H^d_{\fm}(I^t) \to H^d_{\fm}(R) \to 0\]
induced by the short exact sequence $0\to I^t\to R\to R/I^t\to 0$. Since $H^{d-1}_{\fm}(R)$ is finitely generated, there exists an integer $\ell_0$ such that $H^{d-1}_{\fm}(R)_{<\ell_0}=0$. Let $L$ denote the image of $f$ and $M$ denote the image of $g$. (Note that $M$ depends on $t$.) Since $L$ is a graded quotient of $H^{d-1}_{\fm}(R)$, we have $L_{<\ell_0}=0$. The short exact sequence $0\to L\to H^{d-1}_{\fm}(R/I^t) \to M\to 0$ induces an exact sequence
\[0\to \Soc(L)\to \Soc(H^{d-1}_{\fm}(R/I^t))\to \Soc(M)\to \Ext^1_R(k,L)\]
Lemma \ref{beg for Ext with k} implies that $\Ext^1_R(k,L)_{<\ell_1}=0$ for an integer $\ell_1$. Hence the statement (a) is equivalent to
\begin{center}
 {\it (*) there is a constant $c_1$ such that $\beg(M)\geq -c_1t$ for all $t\geq 1$.}
 \end{center} 
The short exact sequence $0\to M\to H^d_{\fm}(I^t) \to H^d_{\fm}(R) \to 0$ induces an exact sequence
\[0\to \Soc(M)\to \Soc(H^d_{\fm}(I^t)) \to \Soc(H^d_{\fm}(R))\]
Since $\Soc(H^d_{\fm}(R))$ is bounded from below by a constant (independent of $t$), it follows that the statement $(*)$ is equivalent to the statement (b). This completes the proof.
\end{proof}

\section{An application to gauge-boundedness}
\label{gauge-bounded}
Let $R$ be an $F$-finite noetherian commutative ring of characteristic $p$. The full Cartier algebra of $R$ is defined as
\[\cC^R:=\oplus_{e\geq 0}\cC^R_e,\quad {\rm where\ }\cC^R_e=\Hom_R(F^e_*R,R).\]
The investigation of singularities and test ideals in characteristic $p$ using $\cC^R$ was initiated by Schwede in \cite{SchwedeFAdjunction,SchwedeTestIdealsNonQGorenstein} and later the idea was refined by Blickle in \cite{BlickleTestIdealspInverseMaps}. The Cartier algebra $\cC^R$ comes with a non-commutative ring structure (see \cite[\S2]{BlickleTestIdealspInverseMaps}). Generally speaking, the ring structure of $\cC^R$ is rather complex. A notion of {\it gauge bounded} introduced in \cite{BlickleTestIdealspInverseMaps} can be used to control the complexity of $\cC^R$. 

\begin{remark}
Let $S=k[x_1,\dots,x_n]$ where $k$ is an $F$-finite field and let $R=S/I$. One can build a filtration of $R$ as follows\footnote{As discussed in \cite[\S4]{BlickleTestIdealspInverseMaps}, there are at least two different filtrations of $S$ one can use: one is filtered by the total degree -- the one we use in this section and in \cite{BlickleSchwedeTakagiZhang}; the other is filtered by the max degree of any variable.} . Set $R_m$ to be the finite dimensional $k$-vector space of $R$ spanned by the images of the monomials of degree at most $m$:
\[x^{e_1}_1\cdots x^{e_n}_n,\quad {\rm for\ }\sum_ie_i\leq m.\]
For each $r\in R$, one sets $\delta(r)=m$ iff $r\in R_m\backslash R_{m-1}$ and calls $\delta$ a {\it gauge} for $R$.

The full Cartier algebra $\cC^R$ is {\it gauge bounded} if there is a constant $c$ and elements $\varphi_{e,j}\in \cC^R_e$ (generating $\cC^R_e$ as a left $R$-module) such that 
\[\delta(\varphi_{e,j}(r))\leq \frac{\delta(r)}{p^e}+c,\quad {\rm for\ each\ }e\ {\rm and\ }j.\]

If $\cC^R$ is gauge bounded, then for each ideal $\fa$ of $R$, the $F$-jumping numbers of $\tau(R, \fa^t)$ are a subset of the real numbers with no limit points; in particular, they form a discrete set. This is \cite[Corollary 4.19]{BlickleTestIdealspInverseMaps}.
\end{remark}

We record the following sufficient condition for $\cC^R$ to be gauge bounded, which is most relevant to this section.

\begin{lemma}
\label{linear growth in degree implies gauge bounded}
Let $R$ be a standard graded algebra over a field $k$ of characteristic $p$. If there is a constant $c$ such that such that the maximal degree of any minimal generator of $\Hom_R(F_*^eR,R)$ is at most $c$ for all $e\geq 1$, then $\cC^R$ is gauge-bounded.
\end{lemma}
\begin{proof}
Write $R=S/\fa$ where $S$ is a polynomial ring and $\fa$ is a homogeneous ideal. Then \cite[Lemma 1.6]{FedderFPurity} says that
\[\Hom_R(F^e_*R,R)\cong F^e_*(\frac{(\fa^{[p^e]}:\fa)}{\fa^{[p^e]}}).\]
Therefore our lemma is equivalent to \cite[Lemma 2.2]{KatzmanZhangRegularityandFJumpingNumbers}. 
\end{proof}

Let $R$ be a standard graded integral domain over an $F$-finite field $k$ of characteristic $p$. By Proposition \ref{graded canonical ideal}, $R$ admits a graded canonical ideal $\omega$ such that a degree-shift $\omega(\beg(\omega)+a(R))$ is the graded canonical module. To ease our notation, we will set $a:=\beg(\omega)+a(R)$, {\it i.e.} $\omega(a)$ is the graded canonical module of $R$.

The following theorem is crucial to our application to gauge-boundedness.
\begin{theorem}
\label{gauge same as least socle degree}
Let $R$ be a standard graded integral domain over an $F$-finite field $k$ of characteristic $p$ and $\omega$ be the canonical ideal. Assume that $R$ satisfies $(S_2)$ condition. Then there exists a constant $c$ such that the maximal degree of any minimal generator of $\Hom_R(F_*^eR,R)$ is at most $c$ for all $e\geq 1$ if and only if  
\[\beg(\Soc(H^d_{\fm}(\omega^{[p^e]})))\geq -(c-a)p^e\] 
for all $e\geq 1$, where $d=\dim(R)$.
\end{theorem}
\begin{proof}
We have degree-preserving isomorphisms:
\[\Hom_R(F^e_*R,R)\cong \Hom_R(F_*^eR,\Hom_R(\omega,\omega))\cong \Hom_R(F^e_*R\otimes_R\omega,\omega).\] 
Let $S$ denote the polynomial ring $k[x_1,\dots,x_n]$. Under our hypothesis, $F^e_*R\otimes_RJ$ is finitely generated. Let $ \cF_{\bullet}\to \Hom_R(F^e_*R,R)\to 0$ be the minimal graded resolution of $\Hom_R(F^e_*R,R)$. By Remark \ref{Matlis dual gives essential extension}, $\Hom_R(F^e_*R,R)^{\vee}\to \cF_{0}^{\vee}$ is an essential extension. Set $d=\dim(R)$. Then we have degree-preserving isomorphisms:
\begin{align*}
H^d_{\fm}(F^e_*R\otimes_R\omega)&\cong \Hom_R(F^e_*R\otimes_R\omega,\omega(a))^{\vee}\\
&\cong \Hom_R(F^e_*R\otimes_R\omega,\omega)^{\vee}(-a)\\
&\cong \Hom_R(F^e_*R, \Hom_R(\omega,\omega))^{\vee}(-a)\\
&\cong \Hom_R(F^e_*R,R)^{\vee}(-a)
\end{align*}
Therefore the resulted degree-preserving $R$-linear map $H^d_{\fm}(F^e_*R\otimes_R\omega)\to \cF_{0}^{\vee}(-a)$ is an essential extension. Write $\cF_0=\oplus_{j=1}^{\mu_e} R(-\alpha_{ej})$ where $\mu_e$ is the minimal number of generators of $\Hom_R(F^e_*R,R)$. Then $\alpha:=\max\{\alpha_{ej}\}$ is the maximal degree of a minimal generator of $\Hom_R(F^e_*R,R)$. With this presentation of $\cF_0$, it is evident that 
\[\cF_{0}^{\vee}=\Hom_R(\cF_{0},\E)\cong \oplus_{j=1}^{\mu_e} \E(\alpha_{ej}).\]
Consequently, the degree-preserving $R$-linear map $H^d_{\fm}(F^e_*R\otimes_R\omega)\to \oplus_{j=1}^{\mu_e}\E(\alpha_{ej}-a)$ is an essential extension. Since both $H^d_{\fm}(F^e_*R\otimes_R\omega)$ and $\E$ are artinian, we must have
\[\Soc(H^d_{\fm}(F^e_*R\otimes_R\omega))=\Soc(\oplus_{j=1}^{\mu_e}\E(\alpha_{ej}-a))=\oplus_{j=1}^{\mu_e}k(\alpha_{ej}-a)\]
where we use the fact that $\Soc(\E)=k$ and it lives in degree 0. Therefore,
\[\beg(\Soc(H^d_{\fm}(F^e_*R\otimes_R\omega)))=\min \{-\alpha_{ej}+a\}=-\max\{\alpha_{ej}-a\}.\]
The short exact sequence $0\to \omega\to R\to R/\omega\to 0$ induces an exact sequence $0\to \Tor_1(F^e_*R, R/\omega)\to F^e_*R\otimes_R\omega\to F^e_*\omega^{[p^e]}\to 0$ which in turn induces an exact sequence of local cohomology:
\[H^{d}_{\fm}(\Tor_1(F^e_*R, R/\omega))\to H^d_{\fm}(F^e_*R\otimes_R\omega)\to H^d_{\fm}(F^e_*\omega^{[p^e]})\to 0\]
Since $\Tor_1(F^e_*R, R/\omega)$ is killed by $\omega$, its dimension is less than $d$. Thus, $H^{d}_{\fm}(\Tor_1(F^e_*R, R/\omega))=0$. This implies $H^d_{\fm}(F^e_*R\otimes_R\omega)\xrightarrow{\sim} H^d_{\fm}(F^e_*\omega^{[p^e]})$. Therefore
\[\beg(\Soc(H^d_{\fm}(\omega^{[p^e]})))=p^e \beg(\Soc(H^d_{\fm}(F^e_*\omega^{[p^e]})))=p^e(-\max\{\alpha_{ej}-a\}).\]
This shows that there is a constant $c$ such that $\max\{\alpha_{ej}\}\leq c$ for all $e\geq 1$ if and only if $\beg(\Soc(H^d_{\fm}(\omega^{[p^e]})))\geq -(c-a)p^e$ for all $e\geq 1$.
\end{proof}

\begin{theorem}
\label{thm: linear bound on socle degree}
Let $R$ be a $d$-dimensional standard graded integral domain over an $F$-finite field $k$ of characteristic $p$ and let $\omega$ denote its graded canonical ideal. Assume that $R$ satisfies $(S_2)$ condition and that $R_{\fp}$ is Gorenstein for each non-maximal prime ideal $\fp$. Then there exists a constant $c$ such that 
\begin{equation}
\label{linear bound for socle of Frobenius power of canonical}
\beg(\Soc(H^d_{\fm}(\omega^{[p^e]})))\geq -cp^e
\end{equation}
for all $e\geq 1$.
\end{theorem}
\begin{proof}
Since $R_{\fp}$ is Gorenstein for each non-maximal prime ideal $\fp$, we have $\omega_{\fp}$ is a principal ideal for each non-maximal prime ideal $\fp$. It follows that $\omega^{[p^e]}_{\fp}=\omega^{p^e}_{\fp}$ for each non-maximal prime ideal $\fp$. Therefore $\dim(\omega^{p^e}/\omega^{[p^e]})=0$. Hence $H^d_{\fm}(\omega^{p^e})=H^d_{\fm}(\omega^{[p^e]})$. Hence (\ref{linear bound for socle of Frobenius power of canonical}) is equivalent to 
\begin{equation}
\label{linear bound for socle of power of canonical}
\beg(\Soc(H^d_{\fm}(\omega^{p^e})))\geq -cp^e
\end{equation}
for all $e\geq 1$. 

Since $R_{\fp}$ is Gorenstein for each non-maximal prime ideal $\fp$, it follows $R$ is generalized Cohen-Macaulay and hence $H^{d-1}_{\fm}(R)$ is finitely generated. A combination of Lemma \ref{linking ideal and quotient} and Corollary \ref{corollary on socle of top local cohomology} completes the proof. 
\end{proof}

The following theorem now follows immediately from Lemma \ref{linear growth in degree implies gauge bounded}, Theorem \ref{thm: linear bound on socle degree} and Theorem \ref{gauge same as least socle degree}.
\begin{theorem}
\label{isolated Gorenstein gauge-bounded}
Let $R$ be a $d$-dimensional standard graded integral domain over an $F$-finite field $k$ of characteristic $p$. Assume that $R$ satisfies $(S_2)$ condition and that $R_{\fp}$ is Gorenstein for each non-maximal prime ideal $\fp$. Then $\cC^R$ is gauge bounded.
\end{theorem}

\subsection*{Acknowledgement} The author thanks Linquan Ma for pointing out a reference on canonical ideals and Mordechai Katzman for comments on a preliminary draft. This article is partially motivated by \cite{BBLSZ3}, a joint paper with Bhargav Bhatt, Manuel Blickle, Gennady Lyubeznik, and Anurag Singh. The author is grateful to them for the collaborations and mathematical discussions over the years. The author also wants to thank the referee for useful suggestions.

\bibliographystyle{skalpha}
\bibliography{lineargrowth}

\end{document}